\documentclass{article}

\usepackage{amsmath,amssymb,amsthm}

\newtheorem{theorem}{Theorem}[section]
\newtheorem{lemma}[theorem]{Lemma}
\newtheorem{corollary}[theorem]{Corollary}
\newtheorem{proposition}[theorem]{Proposition}
\theoremstyle{definition}
\newtheorem{definition}[theorem]{Definition}
\newtheorem{example}[theorem]{Example}
\theoremstyle{remark}
\newtheorem{remark}[theorem]{Remark}
\everymath{\displaystyle}

\begin{document}

\title{Necessary condition for an Euler--Lagrange equation on time scales\thanks{This is a preprint 
of a paper whose final and definite form is \emph{Abstract and Applied Analysis} 2014 (2014), 
Article ID 631281, http://dx.doi.org/10.1155/2014/631281. 
Submitted 28/Jan/2014; Revised 12/March/2014; Accepted: 13/March/2014.
Part of first author's Ph.D., which is carried out at the University of Aveiro under
the Doctoral Programme \emph{Mathematics and Applications}
of Universities of Aveiro and Minho.}}

\author{Monika Dryl\\
\texttt{monikadryl@ua.pt}
\and
Delfim F. M. Torres\\
\texttt{delfim@ua.pt}}

\date{Center for Research and Development in Mathematics and Applications (CIDMA),
Department of Mathematics,\\ University of Aveiro, 3810--193 Aveiro, Portugal}

\maketitle


\begin{abstract}
We prove a necessary condition for a dynamic integro-differential equation
to be an Euler--Lagrange equation. New and interesting results for the discrete
and quantum calculus are obtained as particular cases. An example of a second
order dynamic equation, which is not an Euler--Lagrange equation
on an arbitrary time scale, is given.

\bigskip

\noindent \textbf{Keywords:} time scales, calculus of variations,
Euler--Lagrange equations, self-adjointness, inverse problem.

\bigskip

\noindent \textbf{Mathematics Subject Classification (2010):} 34N05; 49K05; 49N45.
\end{abstract}


\section{Introduction}

The time-scale calculus is a unification of the theories of difference
and differential equations, unifying integral and differential calculus
with the calculus of finite differences, and offering a formalism
for studying hybrid discrete-continuous dynamical systems
\cite{MR1066641,BohnerDEOTS}. It has applications in any field that requires
simultaneous modeling of discrete and continuous data \cite{MR1908825,MBbook2003,MyID:267}.

The study of optimal control problems on arbitrary time scales is a subject under strong
current research \cite{MR3115452,MR3031162}. This is particularly true for the particular,
but rich case, of the calculus of variations on time scales \cite{MR2528200,MR3040923,MR2671876}.
Compared with the direct problem, that establish
dynamic equations of Euler--Lagrange type to the time-scale variational problems,
the inverse problem has not yet been studied in the framework of time scales.
It turns out that there is a simple explanation for the absence of such an inverse general theory for the
time-scale variational calculus: the classical approach relies on the use
of the chain rule, which is not valid in the general context of time scales \cite{BohnerDEOTS}.
To address the problem, a different approach to the subject is needed.

In this paper we introduce a completely different approach to the inverse problem of the calculus
of variations, using an integral perspective instead of the classical differential point of view
\cite{Helmholtz,Davis}. The differential form of equations is often related to dynamics via the time derivative.
The integral form has proved to be successful for proving the existence and uniqueness of solutions,
to study analytical properties of solutions, and to prove coherence of variational embeddings
\cite{Cresson}. Here we show its usefulness with respect to the inverse problem of the calculus of variations.
We prove a necessary condition for an integro-differential equation
on an arbitrary time scale $\mathbb{T}$ to be an Euler--Lagrange equation,
related with a property of self-adjointness (Definition~\ref{def self adj}) of the equation of variation
(Definition~\ref{def 2}) of the given dynamic integro-differential equation.

The text is organized as follows. Section~\ref{sec:Preliminairies} provides all the necessary definitions
and results of the delta-calculus on time scales, which will be used throughout the text.
The main results are proved in Section~\ref{main results}. We present a sufficient condition
of self-adjointness for an integro-differential equation (Lemma~\ref{th self adj}). Using this property,
we prove a necessary condition for a general (non-classical) inverse problem of the calculus of variations
on an arbitrary time scale (Theorem~\ref{th self adj E-L}). As a result, we obtain a useful tool
to identify integro-differential equations which are not Euler--Lagrange equations
(Remark~\ref{not E-L}). To illustrate the method, we give a second order dynamic equation
on time scales which is not an Euler--Lagrange equation (Example~\ref{ex:NEL}).
Next we apply Theorem~\ref{th self adj E-L} to the particular cases of time scales
$\mathbb{T}\in\lbrace \mathbb{R},h\mathbb{Z},\overline{q^{\mathbb{Z}}}\rbrace$, $h > 0$, $q>1$
(Corollaries \ref{cor R}, \ref{cor hZ}, and \ref{cor qZ}). In Section~\ref{final remarks}
some final remarks are presented. We begin by proving the equivalence between
an integro-differential equation and a second order dynamic equation
(Proposition~\ref{prop 1}). Then we show that, due to lack of a chain rule in an arbitrary time scale,
it is impossible to obtain an equivalence between equations of variation in integral and differential forms.
This is in contrast with the classical case $\mathbb{T}=\mathbb{R}$, where such equivalence holds (Proposition~\ref{prop}).


\section{Preliminaries}
\label{sec:Preliminairies}

In this section we introduce basic definitions and theorems that will be useful in the sequel.
For more results concerning the theory of time scales we refer the reader to the books
\cite{BohnerDEOTS,MBbook2003}.

\begin{definition}[\textrm{e.g.}, Section 2.1 of \cite{TorresDeltaNabla}]
A time scale $\mathbb{T}$ is an arbitrary nonempty closed subset of $\mathbb{R}$.
Given a time scale $\mathbb{T}$, the forward jump operator
$\sigma:\mathbb{T}\rightarrow \mathbb{T}$
is defined by $\sigma(t):=\inf\lbrace s\in\mathbb{T}: s>t\rbrace$
for $t\neq \sup\mathbb{T}$ and $\sigma(\sup\mathbb{T}) := \sup\mathbb{T}$
if $\sup\mathbb{T}<+\infty$. Similarly, the backward jump operator
$\rho:\mathbb{T}\rightarrow \mathbb{T}$
is defined by $\rho(t):=\sup\lbrace s\in\mathbb{T}: s<t\rbrace$
for $t\neq \inf\mathbb{T}$ and $\rho(\inf\mathbb{T})=\inf\mathbb{T}$
if $\inf\mathbb{T}>-\infty$.
\end{definition}

A point $t \in \mathbb{T}$ is called right-dense, right-scattered,
left-dense or left-scattered if $\sigma(t) = t$, $\sigma(t) > t$,
$\rho(t) = t$, $\rho(t) < t$, respectively.
The forward graininess function
$\mu:\mathbb{T} \rightarrow [0,\infty)$
is defined by $\mu(t):=\sigma(t) -t$.
To simplify the notation, one usually uses $f^{\sigma}(t):=f(\sigma(t))$.

The delta derivative is defined for points from the set
$$
\mathbb{T}^{\kappa} :=
\begin{cases}
\mathbb{T}\setminus\left\{\sup\mathbb{T}\right\}
& \text{ if } \rho(\sup\mathbb{T})<\sup\mathbb{T}<\infty,\\
\mathbb{T}
& \hbox{ otherwise}.
\end{cases}
$$
\begin{definition}[Section 1.1 of \cite{BohnerDEOTS}]
Let $f:\mathbb{T}\rightarrow\mathbb{R}$ and $t\in\mathbb{T}^{\kappa}$.
We define $f^{\Delta}(t)$ to be the number (provided it exists)
with the property that given any $\varepsilon >0$,
there is a neighborhood $U$ of $t$ such that
\begin{equation*}
\left|f^{\sigma}(t)-f(s)-f^{\Delta}(t)\left(\sigma(t)-s\right)\right|
\leq \varepsilon \left|\sigma(t)-s\right| \mbox{ for all }  s\in U.
\end{equation*}
We call $f^{\Delta}(t)$ the delta derivative of $f$ at $t$.
Function $f$ is delta differentiable on $\mathbb{T}^{\kappa}$ provided
$f^{\Delta}(t)$ exists for all $t\in\mathbb{T}^{\kappa}$. Then,
$f^{\Delta}:\mathbb{T}^{\kappa}\rightarrow\mathbb{R}$
is called the delta derivative of $f$ on $\mathbb{T}^{\kappa}$.
\end{definition}

\begin{theorem}[Theorem 1.16 of \cite{BohnerDEOTS}]
\label{rozniczka delta}
Let $f:\mathbb{T} \rightarrow \mathbb{R}$
and $t\in\mathbb{T}^{\kappa}$. If $f$ is continuous at $t$
and $t$ is right-scattered, then $f$ is delta differentiable at $t$ with
$$
f^{\Delta}(t)=\frac{f^{\sigma}(t)-f(t)}{\mu(t)}.
$$
\end{theorem}

\begin{theorem}[Theorem 1.20 of \cite{BohnerDEOTS}]
\label{tw:differpropdelta}
Let $f,g:\mathbb{T} \rightarrow \mathbb{R}$
be delta differentiable at $t\in\mathbb{T^{\kappa}}$. Then,
\begin{enumerate}

\item the sum $f+g:\mathbb{T} \rightarrow \mathbb{R}$ is delta differentiable at $t$ with
\begin{equation*}
(f+g)^{\Delta}(t)=f^{\Delta}(t)+g^{\Delta}(t);
\end{equation*}

\item for any real constant $\alpha$, $\alpha f:\mathbb{T}\rightarrow\mathbb{R}$
is delta differentiable at $t$ with
\begin{equation*}
(\alpha f)^{\Delta}(t)=\alpha f^{\Delta}(t);
\end{equation*}

\item the product $fg:\mathbb{T}\rightarrow\mathbb{R}$
is delta differentiable at $t$ with
\begin{equation*}
(fg)^{\Delta}(t)=f^{\Delta}(t)g(t)+f^{\sigma}(t) g^{\Delta}(t)
=f(t)g^{\Delta}(t)+f^{\Delta}(t)g^{\sigma}(t).
\end{equation*}
\end{enumerate}
\end{theorem}

\begin{theorem}[Theorem 1.16 from \cite{BohnerDEOTS}]
\label{differ}
If $f:\mathbb{T}\longrightarrow \mathbb{R}$ is a delta differentiable
function at $t$, $t\in\mathbb{T}^{\kappa}$, then
$$
f^{\sigma}(t)=f(t)+\mu(t)f^{\Delta}(t).
$$
\end{theorem}

\begin{definition}[Definition~1.58 of \cite{BohnerDEOTS}]
A function $f:\mathbb{T}\rightarrow \mathbb{R}$ is called rd-continuous
provided it is continuous at right-dense points in $\mathbb{T}$ and its
left-sided limits exist (finite) at all left-dense points in $\mathbb{T}$.
\end{definition}

The set of all rd-continuous functions $f:\mathbb{T} \rightarrow \mathbb{R}$
is denoted by $C_{rd} = C_{rd}(\mathbb{T}) = C_{rd}(\mathbb{T},\mathbb{R})$.
The set of functions $f:\mathbb{T} \rightarrow \mathbb{R}$ that are
delta differentiable and whose derivative is rd-continuous is denoted by
$C^{1}_{rd}=C_{rd}^{1}(\mathbb{T})=C^{1}_{rd}(\mathbb{T},\mathbb{R})$.

\begin{definition}[Definition~1.71 of \cite{BohnerDEOTS}]
A function $F:\mathbb{T} \rightarrow \mathbb{R}$ is called
an antiderivative of $f:\mathbb{T} \rightarrow \mathbb{R}$ provided
$F^{\Delta}(t)=f(t)$ for all $t\in\mathbb{T}^{\kappa}$.
\end{definition}

\begin{definition}
Let $\mathbb{T}$ be a time scale and $a,b\in\mathbb{T}$.
If $f:\mathbb{T}^{\kappa} \rightarrow \mathbb{R}$ is a rd-continuous
function and $F:\mathbb{T} \rightarrow \mathbb{R}$
is an antiderivative of $f$, then the delta integral is defined by
$$
\int\limits_{a}^{b} f(t)\Delta t := F(b)-F(a).
$$
\end{definition}

\begin{theorem}[Theorem~1.74 of \cite{BohnerDEOTS}]
Every rd-continuous function $f$ has an antiderivative $F$.
In particular, if $t_{0}\in\mathbb{T}$, then $F$ defined by
$$
F(t):=\int\limits_{t_{0}}^{t} f(\tau)\Delta \tau, \quad t\in\mathbb{T},
$$
is an antiderivative of $f$.
\end{theorem}

\begin{example}
Let $a,b\in\mathbb{T}$ and $f:\mathbb{T} \rightarrow \mathbb{R}$
be rd-continuous. If $\mathbb{T}=\mathbb{R}$, then
\begin{equation*}
\int\limits_{a}^{b}f(t)\Delta t=\int\limits_{a}^{b}f(t)dt,
\end{equation*}
where the integral on the right side is the usual Riemann integral.
If $\mathbb{T}=h\mathbb{Z}$, $h>0$, then
\begin{equation*}
\int\limits_{a}^{b}f(t)\Delta t
=
\begin{cases}
\sum\limits_{k=\frac{a}{h}}^{\frac{b}{h}-1}f(kh)h, & \hbox{ if } a<b, \\
0, & \hbox{ if } a=b,\\
-\sum\limits_{k=\frac{b}{h}}^{\frac{a}{h}-1}f(kh)h, & \hbox{ if } a>b.
\end{cases}
\end{equation*}
If $\mathbb{T}=\overline{q^{\mathbb{Z}}}$, $q>1$, and $a<b$, then
$$
\int\limits_{a}^{b}f(t)\Delta t=(q-1)\sum_{t\in[a,b)\cap\mathbb{T}}t f(t).
$$
\end{example}

\begin{theorem}[Theorem 1.77 from \cite{BohnerDEOTS}]
\label{integration}
If $a,b,c\in\mathbb{T}, \alpha\in\mathbb{R}, \hbox{ and }f,g\in C_{rd}(\mathbb{T})$, then
\begin{enumerate}
\item $\int\limits_{a}^{b}\left[f(t)+g(t)\right]\Delta t
=\int\limits_{a}^{b}f(t)\Delta t+\int\limits_{a}^{b}g(t)\Delta t$;

\item $\int\limits_{a}^{b}\left(\alpha f\right)(t)\Delta t
=\alpha \int\limits_{a}^{b} f(t)\Delta t$;

\item $\int\limits_{a}^{b}f(t)g^{\Delta}(t)\Delta t
=(fg)(b)-(fg)(a)-\int\limits_{a}^{b}f^{\Delta}(t)g^{\sigma}(t)\Delta t$;

\item$\int\limits_{a}^{b}f^{\sigma}(t)g^{\Delta}(t)\Delta t
=(fg)(b)-(fg)(a)-\int\limits_{a}^{b}f^{\Delta}(t)g(t)\Delta t$.
\end{enumerate}
\end{theorem}

For more properties of the delta derivative and delta integral
we refer the reader to \cite{BohnerDEOTS,MBbook2003}.


\section{Main results}
\label{main results}

Our main result (Theorem~\ref{th self adj E-L}) provides a necessary condition for
an integro-differential equation on an arbitrary time scale to be an Euler--Lagrange equation.
For that the notions of self-adjointness (Definition~\ref{def self adj})
and equation of variation (Definition~\ref{def 2}) are essential.
These definitions, in integro-differential form, are new (cf. the notion
of self-adjointness for a dynamic time-scale equation of second order
in \cite[Sec.~4.1]{BohnerDEOTS} and the notion of equation
of variation for a second order differential equation in \cite{Davis}).

\begin{definition}[First order self-adjoint integro-differential equation]
\label{def self adj}
A first order integro-differential dynamic equation is said to be \emph{self-adjoint} if it has the form
\begin{equation}
\label{self adj}
Lu(t)=const, \hbox{ where } Lu(t)=p(t)u^{\Delta}(t)
+\int\limits_{t_{0}}^{t}\left[q(s)u^{\sigma}(s)\right]\Delta s,
\end{equation}
with $p,q\in C_{rd}$, $p\neq 0$ for all $t\in\mathbb{T}$, and $t_0 \in \mathbb{T}$.
\end{definition}
Let $\mathbb{D}$ be the set of all functions $y:\mathbb{T}\longrightarrow \mathbb{R}$
such that $y^{\Delta}:\mathbb{T}^{\kappa}\longrightarrow\mathbb{R}$ is continuous.
A function $y\in\mathbb{D}$ is said to be a solution of \eqref{self adj} provided $Ly(t)=const$
holds for all $t\in\mathbb{T^{\kappa}}$. Along the text we use the operators
$[\cdot]_{\mathbb{T}}$ and $\langle \cdot\rangle_{\mathbb{T}}$ defined by
\begin{equation}
\label{notation 1}
[y]_{\mathbb{T}}(t):=(t,y^{\sigma}(t), y^{\Delta}(t)) \hbox{ and } \langle y\rangle_{\mathbb{T}}(t)
:=(t,y^{\sigma}(t), y^{\Delta}(t), y^{\Delta\Delta}(t)).
\end{equation}

\begin{definition}[Equation of variation]
\label{def 2}
Let
\begin{equation}
\label{integro differ}
H[y]_{\mathbb{T}}(t)+\int\limits_{t_{0}}^{t}G[y]_{\mathbb{T}}(s)\Delta s = const
\end{equation}
be an integro-differential equation on time scales with $\partial_{3}H\neq 0$ and
$t\longrightarrow \partial_{2}F[y](t)$, $t\longrightarrow\partial_{3}F[y](t)\in C_{rd}$
along every curve $y$, where $F \in \{G, H\}$. The \emph{equation of variation}
associated with \eqref{integro differ} is given by
\begin{multline}
\label{eq of var}
\partial_{2}H[u]_{\mathbb{T}}(t)u^{\sigma}(t)+\partial_{3}H[u]_{\mathbb{T}}(t) u^{\Delta}(t)\\
+\int\limits_{t_{0}}^{t}\partial_{2}G[u]_{\mathbb{T}}(s)u^{\sigma}(s)
+\partial_{3}G[u]_{\mathbb{T}}(s) u^{\Delta}(s)\Delta s=0.
\end{multline}
\end{definition}

\begin{remark}
\label{remark 1}
The equation of variation \eqref{eq of var} can be interpreted in the following way.
Assuming $y=y(t,b)$, $b\in\mathbb{R}$, is a one-parameter solution of a given
integro-differential equation \eqref{integro differ}, then
\begin{equation}
\label{integro differ (y,b)}
H(t,y^{\sigma}(t,b), y^{\Delta}(t,b))
+\int\limits_{t_{0}}^{t}G(s,y^{\sigma}(s,b),y^{\Delta}(s,b))\Delta s = const.
\end{equation}
Let $u(t)$ be a particular solution, that is, $u(t)=y(t,\bar{b})$ for a certain $\bar{b}$.
Differentiating \eqref{integro differ (y,b)} with respect to the parameter $b$
and then putting $b=\bar{b}$, we obtain equation \eqref{eq of var}.
\end{remark}

\begin{lemma}[Sufficient condition of self-adjointness]
\label{th self adj}
Let \eqref{integro differ} be a given integro-differential equation. If
\begin{equation}
\label{assumption self-adj}
\partial_{2}H[y]_{\mathbb{T}}(t)+\partial_{3}G[y]_{\mathbb{T}}(t)=0,
\end{equation}
then its equation of variation \eqref{eq of var} is self-adjoint.
\end{lemma}

\begin{proof}
Let us consider a given equation of variation \eqref{eq of var}. Using Theorem~\ref{differ}
and third item of Theorem~\ref{integration}, we expand the two components of the given equation:
$$
\partial_{2}H[u]_{\mathbb{T}}(t)u^{\sigma}(t)
=\partial_{2}H[u]_{\mathbb{T}}(t)\left(u(t)+\mu(t) u^{\Delta}(t)\right),
$$
\begin{multline*}
\int\limits_{t_{0}}^{t}\partial_{3}G[u]_{\mathbb{T}}(s)u^{\Delta}(s)\Delta s\\
=\partial_{3}G[u]_{\mathbb{T}}(t)u(t)-\partial_{3}G[u]_{\mathbb{T}}(t_{0})u(t_{0})
-\int\limits_{t_{0}}^{t}\left[\partial_{3}G[u]_{\mathbb{T}}(s)\right]^{\Delta}u^{\sigma}(s)\Delta s.
\end{multline*}
Hence, equation of variation \eqref{eq of var} can be written in the form
\begin{multline}
\label{eq 2}
\partial_{3}G[u]_{\mathbb{T}}(t_{0})u(t_{0}) =
u^{\Delta}(t)\left[\mu(t)\partial_{2}H[u]_{\mathbb{T}}(t)+\partial_{3}H[u]_{\mathbb{T}}(t)\right]\\
+\int\limits_{t_{0}}^{t}u^{\sigma}(s)\left[\partial_{2}G[u]_{\mathbb{T}}(s)
-(\partial_{3}G[u]_{\mathbb{T}}(s))^{\Delta}\right]\Delta s\\
+u(t)\left(\partial_{2}H[u]_{\mathbb{T}}(t)+\partial_{3}G[u]_{\mathbb{T}}(t)\right).
\end{multline}
If \eqref{assumption self-adj} holds, then \eqref{eq 2}
is a particular case of \eqref{self adj} with
\begin{gather*}
p(t)=\mu(t)\partial_{2}H[u]_{\mathbb{T}}(t)+\partial_{3}H[u]_{\mathbb{T}}(t),\\
q(s)=\partial_{2}G[u]_{\mathbb{T}}(s)-(\partial_{3}G[u]_{\mathbb{T}}(s))^{\Delta},\\
\partial_{3}G[u]_{\mathbb{T}}(t_{0})u(t_{0})=const.
\end{gather*}
This concludes the proof.
\end{proof}

\begin{theorem}[Necessary condition for an Euler--Lagrange equation in integral form]
\label{th self adj E-L}
Let $\mathbb{T}$ be and arbitrary time scale and
\begin{equation}
\label{eq:mEL}
H(t,y^{\sigma}(t),y^{\Delta}(t))+\int\limits_{t_{0}}^{t}G(s,y^{\sigma}(s),y^{\Delta}(s))\Delta s=const
\end{equation}
be a given integro-differential equation. If \eqref{eq:mEL} is to be an Euler--Lagrange equation,
then its equation of variation \eqref{eq of var} is self-adjoint,
in the sense of Definition~\ref{def self adj}.
\end{theorem}

\begin{proof}
Assume \eqref{eq:mEL} is the Euler--Lagrange equation of the variational functional
\begin{equation}
\label{var funct}
\mathcal{I}(y)=\int\limits_{t_{0}}^{t_{1}} L(t,y^{\sigma}(t),y^{\Delta}(t))\Delta t,
\end{equation}
where $L \in C^2$. Since the Euler--Lagrange
equation in integral form of \eqref{var funct} is given by
\begin{equation*}
\partial_{3}L[y](t)+\int\limits_{t_{0}}^{t} - \partial_{2}L[y](s)\Delta s=const
\end{equation*}
(cf. \cite{Cresson,FerreiraMalinowskaTorres,FerreiraTorres}), we conclude that
$H[y](t)=\partial_{3}L[y](t)$ and $G[y](s)=-\partial_{2}L[y](s)$.
Having in mind that
\begin{itemize}
\item $\partial_{2}H=\partial_{2}(\partial_{3}L)$,
$\partial_{3}H=\partial_{3}(\partial_{3}L)=\partial_{3}^{2}L$,

\item $\partial_{2}G=\partial_{2}(-\partial_{2}L)=-\partial_{2}^{2}L$,
$\partial_{3}G=\partial_{3}(-\partial_{2}L)=-\partial_{3}\partial_{2}L$,
\end{itemize}
it follows from Schwarz's theorem,
$\partial_{2}\partial_{3}L=\partial_{3}\partial_{2}L$, that
$$
\partial_{2}H[y](t)+\partial_{3}G[y](t)=0.
$$
We conclude from Lemma~\ref{th self adj} that the equation
of variation \eqref{eq:mEL} is self-adjoint.
\end{proof}

\begin{remark}
\label{not E-L}
In practical terms, Theorem~\ref{th self adj E-L} is useful to identify equations
which are not Euler--Lagrange equations: if the equation of variation \eqref{eq of var}
of a given dynamic equation \eqref{integro differ} is not self-adjoint,
then we conclude that \eqref{integro differ} is not an Euler--Lagrange equation.
\end{remark}

\begin{remark}[Self-adjointness for a second order differential equation]
\label{remark 2}
Let $p$ be delta-differentiable in Definition~\ref{def self adj} and $u\in C_{rd}^{2}$.
Then, by differentiating \eqref{self adj}, one obtains
a second-order self-adjoint dynamic equation
\begin{equation*}
p^{\sigma}(t)u^{\Delta\Delta}(t)+p^{\Delta}(t)u^{\Delta}(t)+q(t)u^{\sigma}(t)=0
\end{equation*}
or
\begin{equation*}
p(t)u^{\Delta\Delta}(t)+p^{\Delta}(t)u^{\Delta\sigma}(t)+q(t)u^{\sigma}(t)=0
\end{equation*}
with $q\in C_{rd}$ and $p\in C_{rd}^{1}$ and $p\neq 0$ for all $t\in\mathbb{T}$.
\end{remark}

Now we present an example of a second order differential equation
on time scales which is not an Euler--Lagrange equation.

\begin{example}
\label{ex:NEL}
Let us consider the following second order dynamic equation
in an arbitrary time scale $\mathbb{T}$:
\begin{equation}
\label{ex 1}
y^{\Delta\Delta}(t)+y^{\Delta}(t)-t=0.
\end{equation}
We may write this equation \eqref{ex 1} in integro-differential form \eqref{integro differ}:
\begin{equation}
\label{ex 2}
y^{\Delta}(t)+\int\limits_{t_{0}}^{t}\left(y^{\Delta}(s)-s\right)\Delta s=const,
\end{equation}
where $H[y]_{\mathbb{T}}(t)=y^{\Delta}(t)$
and $G[y]_{\mathbb{T}}(t)=y^{\Delta}(t)-t$. Because
\begin{equation*}
\partial_{2}H[y]_{\mathbb{T}}(t)=\partial_{2}G[y]_{\mathbb{T}}(t)=0,
\quad \partial_{3}H[y]_{\mathbb{T}}(t)=\partial_{3}G[y]_{\mathbb{T}}(t)=1,
\end{equation*}
then the equation of variation associated with \eqref{ex 2} is given by
\begin{equation}
\label{eg 5}
u^{\Delta}(t)+\int\limits_{t_{0}}^{t}u^{\Delta}(s)\Delta s=0 \iff u^{\Delta}(t)+u(t)=u(t_{0}).
\end{equation}
We may notice that equation \eqref{eg 5} cannot be written in form \eqref{self adj}, hence, it is not self-adjoint.
Indeed, notice that \eqref{eg 5} is a first-order dynamic equation while from Remark~\ref{remark 2}
one obtains a second-order dynamic equation. Following Theorem~\ref{th self adj E-L}
(see Remark~\ref{not E-L}), we conclude that equation \eqref{ex 1} is not an Euler--Lagrange equation.
\end{example}

Now we consider the particular case of Theorem~\ref{th self adj E-L} when $\mathbb{T}=\mathbb{R}$
and $y\in C^{2}([t_{0},t_{1}];\mathbb{R})$. In this case our operator $[\cdot]_{\mathbb{T}}$
of \eqref{notation 1} has the form  $[ y]_{\mathbb{R}}(t)=(t,y(t),y'(t))$,
while condition \eqref{self adj} can be written as
\begin{equation}
\label{self adj R}
p(t)u'(t)+\int\limits_{t_{0}}^{t}q(s)u(s)ds=const.
\end{equation}

\begin{corollary}
\label{cor R}
If a given integro-differential equation
$$
H(t,y(t),y'(t))+\int\limits_{t_{0}}^{t}G(s,y(s),y'(s))ds=const
$$
is to be the Euler--Lagrange equation of a variational problem
\begin{equation*}
\mathcal{I}(y)=\int\limits_{t_{0}}^{t_{1}} L(t,y(t),y'(t))dt
\end{equation*}
(cf., e.g., \cite{MR2004181}), then its equation of variation
$$
\partial_{2}H[u]_{\mathbb{R}}(t)u(t)+\partial_{3}H[u]_{\mathbb{R}}(t)u'(t)
+\int\limits_{t_{0}}^{t}\partial_{2}G[u]_{\mathbb{R}}(s)u(s)+\partial_{3}G[u]_{\mathbb{R}}(s)u'(s)ds=0
$$
must be self-adjoint, in the sense of Definition~\ref{def self adj}
with \eqref{self adj} given by \eqref{self adj R}.
\end{corollary}

\begin{proof}
Follows from Theorem~\ref{th self adj E-L} with $\mathbb{T}=\mathbb{R}$.
\end{proof}

Now we consider the particular case of Theorem~\ref{th self adj E-L}
when $\mathbb{T}=h\mathbb{Z}$, $h > 0$. In this case our operator
$[\cdot]_{\mathbb{T}}$ of \eqref{notation 1} has the form
$$
[y]_{h\mathbb{Z}}(t)=(t,y(t+h),\Delta_{h} y(t))=:[y]_{h}(t),
$$
where
$$
\Delta_{h}y(t)=\frac{y(t+h)-y(t)}{h}.
$$
For $\mathbb{T}=h\mathbb{Z}$, $h > 0$, condition \eqref{self adj} can be written as
\begin{equation}
\label{self adj hZ}
p(t)\Delta_{h}u(t)
+\sum\limits_{k=\frac{t_{0}}{h}}^{\frac{t}{h}-1}hq(kh)u(kh+h)
=const.
\end{equation}

\begin{corollary}
\label{cor hZ}
If a given difference equation
$$
H(t,y(t+h),\Delta_{h} y(t))
+\sum\limits_{k=\frac{t_{0}}{h}}^{\frac{t}{h}-1}
h G\left(kh,y(kh+h),\Delta_{h} y(kh)\right)=const
$$
is to be the Euler--Lagrange equation of a discrete variational problem
\begin{equation*}
\mathcal{I}(y)
=\sum\limits_{k=\frac{t_{0}}{h}}^{\frac{t_{1}}{h}-1}
h L\left(kh,y(kh+h),\Delta_{h} y(kh)\right)
\end{equation*}
(cf., e.g., \cite{MyID:179}), then its equation of variation
\begin{multline*}
\partial_{2}H[u]_{h}(t)u(t+h)+\partial_{3}H[u]_{h}(t)\Delta_{h}u(t)\\
+h\sum\limits_{k=\frac{t_{0}}{h}}^{\frac{t}{h}-1}
\partial_{2}\left(G[u]_{h}(kh)u(kh+h)+\partial_{3}G[u]_{h}(kh)\Delta_{h}u(kh)\right)=0
\end{multline*}
is self-adjoint, in the sense of Definition~\ref{def self adj}
with \eqref{self adj} given by \eqref{self adj hZ}.
\end{corollary}

\begin{proof}
Follows from Theorem~\ref{th self adj E-L}
with $\mathbb{T}=h\mathbb{Z}$.
\end{proof}

Finally, let us consider the particular case of Theorem~\ref{th self adj E-L}
when $\mathbb{T}=\overline{q^{\mathbb{Z}}}=q^{\mathbb{Z}}\cup\left\{0\right\}$,
where $q^{\mathbb{Z}}=\left\{q^{k}: k\in\mathbb{Z}, \, q>1\right\}$. In this case operator
$[\cdot]_{\mathbb{T}}$ of \eqref{notation 1} has the form
$[y]_{q^{\mathbb{Z}}}(t)=(t,y(qt),\Delta_{q} y(t))=:[y]_{q}(t)$, where
$$
\Delta_{q}y(t)=\frac{y(qt)-y(t)}{(q-1)t}.
$$
For $\mathbb{T}=\overline{q^{\mathbb{Z}}}$, $q>1$,
condition \eqref{self adj} can be written as
\begin{equation}
\label{self adj qZ}
p(t)\Delta_{q}u(t)+ (q-1)\sum\limits_{s\in [t_{0},t) \cap \mathbb{T}}
s r(s) u(qs) = const
\end{equation}
(cf., e.g., \cite{Rahmat}), where we use notation $r(t)$ instead of $q(t)$ in order
to avoid confusion between the $q=const$ that defines the time scale
and function $q(t)$ of \eqref{self adj}.

\begin{corollary}
\label{cor qZ}
If a given $q$-equation
$$
H(t,y(qt),\Delta_{q} y(t))+(q-1)\sum\limits_{s\in [t_{0},t)
\cap\mathbb{T}}sG(s,y(qs),\Delta_{q}y(s))=const,
$$
$q>1$, is to be the Euler--Lagrange equation of a variational problem
\begin{equation*}
\mathcal{I}(y)=(q-1)\sum\limits_{t\in [t_{0},t_{1})
\cap\mathbb{T}}tL(t,y(qt),\Delta_{q}y(t)),
\end{equation*}
$t_{0}, t_{1}\in \overline{q^{\mathbb{Z}}}$, then its equation of variation
\begin{multline*}
\partial_{2}H[u]_{q}(t)u(qt)+\partial_{3}H[u]_{q}(t)\Delta_{q}u(t)\\
+(q-1)\sum\limits_{s\in [t_{0},t)\cap\mathbb{T}}s\left(
\partial_{2}G[u]_{q}(s)u(qs)+\partial_{3}G[u]_{q}(s)\Delta_{q}u(s)  \right)=0
\end{multline*}
is self-adjoint, in the sense of Definition~\ref{def self adj}
with \eqref{self adj} given by \eqref{self adj qZ}.
\end{corollary}

\begin{proof}
Choose $\mathbb{T}=\overline{q^{\mathbb{Z}}}$ in Theorem~\ref{th self adj E-L}.
\end{proof}

The reader interested in the study of Euler--Lagrange equations for problems
of the $q$-variational calculus is referred
to \cite{FerreiraTorres,MyID:266,MR2966852} and references therein.


\section{Discussion}
\label{final remarks}

In an arbitrary time scale $\mathbb{T}$, it is easy to show equivalence between
the integro-differential equation \eqref{integro differ} and the second order differential
equation \eqref{eq 9} below (Proposition~\ref{prop 1}). However, when we consider equations
of variations of them, we notice that it is impossible to prove an equivalence between
them in an arbitrary time scale. This impossibility is true even in the discrete time scale
$\mathbb{Z}$. The main reason is the lack of chain rule on time scales
(\cite[Example~1.85]{BohnerDEOTS}). However, in $\mathbb{T}=\mathbb{R}$
we can present this equivalence (Proposition~\ref{prop}).

\begin{proposition}
\label{prop 1}
The integro-differential equation \eqref{integro differ}
is equivalent to the second order delta differential equation
\begin{equation}
\label{eq 9}
W(t,y^{\sigma}(t), y^{\Delta}(t), y^{\Delta\Delta}(t))=0.
\end{equation}

\end{proposition}
\begin{proof}
Let \eqref{eq 9} be a given second order differential equation.
We may write it as a sum of two components
\begin{equation}
\label{eq 23}
W\langle y\rangle_{\mathbb{T}} (t)=F\langle y\rangle_{\mathbb{T}} (t)+G[y]_{\mathbb{T}}(t)=0.
\end{equation}
Let $F\langle y\rangle_{\mathbb{T}} =H^{\Delta}[y]_{\mathbb{T}}$. Then,
\begin{equation}
\label{eq 14}
H^{\Delta}(t,y^{\sigma}(t),y^{\Delta}(t))+G(t,y^{\sigma}(t),y^{\Delta}(t))=0.
\end{equation}
Integrating both sides of equation \eqref{eq 14} from $t_{0}$ to $t$,
we obtain the integro-differential equation \eqref{integro differ}.
\end{proof}

Let $\mathbb{T}$ be a time scale such that $\mu$ is delta differentiable.
The equation of variation of a second order differential equation \eqref{eq 9} is given by
\begin{equation}
\label{eq var differ}
\partial_{4}W\langle u\rangle_{\mathbb{T}} (t)u^{\Delta\Delta}(t)
+\partial_{3}W\langle u\rangle_{\mathbb{T}} (t) u^{\Delta}(t)
+\partial_{2}W\langle u\rangle_{\mathbb{T}} (t) u^{\sigma}(t)=0.
\end{equation}
Equation~\eqref{eq var differ} is obtained by using
the method presented in Remark~\ref{remark 1}.

In an arbitrary time scale it is impossible to prove the equivalence between
the equation of variation \eqref{eq of var} and \eqref{eq var differ}. Indeed,
after differentiating both sides of equation \eqref{eq of var} and using the
product rule given by~Theorem \ref{tw:differpropdelta}, we have
\begin{multline}
\label{eq 15}
\partial_{2}H[u]_{\mathbb{T}}(t)u^{\sigma\Delta}(t)
+\partial_{2}H^{\Delta}[u]_{\mathbb{T}}(t)u^{\sigma\sigma}(t)
+\partial_{3}H[u]_{\mathbb{T}}(t)u^{\Delta\Delta}(t)\\
+\partial_{3}H^{\Delta}[u]_{\mathbb{T}}(t)u^{\Delta\sigma}(t)
+\partial_{2}G[u]_{\mathbb{T}}(t)u^{\sigma}(t)
+\partial_{3}G[u]_{\mathbb{T}}(t)u^{\Delta}(t)=0.
\end{multline}
The direct calculations
\begin{itemize}
\item $\partial_{2}H[u]_{\mathbb{T}}(t)u^{\sigma\Delta}(t)
=\partial_{2}H[u]_{\mathbb{T}}(t)(u^{\Delta}(t)+\mu^{\Delta}(t)u^{\Delta}(t)+\mu^{\sigma}(t)u^{\Delta\Delta}(t))$,

\item $\partial_{2}H^{\Delta}[u]_{\mathbb{T}}(t)u^{\sigma\sigma}(t)
=\partial_{2}H^{\Delta}[u]_{\mathbb{T}}(t)(u^{\sigma}(t)+\mu^{\sigma}(t)u^{\Delta}(t)
+\mu(t)\mu^{\sigma}(t)u^{\Delta\Delta}(t))$,

\item $\partial_{3}H^{\Delta}[u]_{\mathbb{T}}(t)u^{\Delta\sigma}(t)
=\partial_{3}H^{\Delta}[u]_{\mathbb{T}}(t)(u^{\Delta}(t)+\mu(t) u^{\Delta\Delta}(t))$,
\end{itemize}
allow us to write the equation \eqref{eq 15} in form
\begin{multline*}
\biggl[\mu^{\sigma}(t)\partial_{2}H[u]_{\mathbb{T}}(t)+\mu(t)\mu^{\sigma}(t)\partial_{2}H^{\Delta}[u]_{\mathbb{T}}(t)\\
+\partial_{3}H[u]_{\mathbb{T}}(t)+\mu(t)\partial_{3}H^{\Delta}[u]_{\mathbb{T}}(t)\biggr] u^{\Delta\Delta}(t)\\
+\biggl[\partial_{2}H[u]_{\mathbb{T}}(t)+(\mu(t)\partial_{2}H[u]_{\mathbb{T}}(t))^{\Delta}
+\partial_{3}H^{\Delta}[u]_{\mathbb{T}}(t)+\partial_{3}G[u]_{\mathbb{T}}(t)\biggr]u^{\Delta}(t)\\
+\biggl[\partial_{2}H^{\Delta}[u]_{\mathbb{T}}(t)+\partial_{2}G[u]_{\mathbb{T}}(t)\biggr] u^{\sigma}(t) =0,
\end{multline*}
that is, using Theorem~\ref{differ},
\begin{multline}
\label{eq 20}
u^{\Delta\Delta}(t)\left[\mu(t)\partial_{2}H[u]_{\mathbb{T}}(t)+\partial_{3}H[u]_{\mathbb{T}}(t)\right]^{\sigma}\\
+u^{\Delta}(t)\left[\partial_{2}H[u]_{\mathbb{T}}(t)+(\mu(t)\partial_{2}H[u]_{\mathbb{T}}(t))^{\Delta}
+\partial_{3}H^{\Delta}[u]_{\mathbb{T}}(t) + \partial_{3}G[u]_{\mathbb{T}}(t)\right]\\
+u^{\sigma}(t)\left[\partial_{2}H^{\Delta}[u]_{\mathbb{T}}(t)+\partial_{2}G[u]_{\mathbb{T}}(t)\right]=0.
\end{multline}
We are not able to prove that the coefficients of equation \eqref{eq 20} are the same as in \eqref{eq var differ},
respectively. This is due to the fact that we cannot find the partial derivatives of \eqref{eq 9}, that is,
$\partial_{4}W\langle u\rangle_{\mathbb{T}}(t)$, $\partial_{3}W\langle u\rangle_{\mathbb{T}}(t)$
and $\partial_{2}W\langle u\rangle_{\mathbb{T}}(t)$, from equation \eqref{eq 14}
because of lack of chain rule in an arbitrary time scale.
The equivalence, however, is true for $\mathbb{T}=\mathbb{R}$.

\begin{proposition}
\label{prop}
The equation of variation
\begin{equation}
\label{eq 22}
\partial_{2}H[u]_{\mathbb{R}}(t)u(t)+\partial_{3}H[u]_{\mathbb{R}}(t)u'(t)
+\int\limits_{t_{0}}^{t}\partial_{2}G[u]_{\mathbb{R}}(s)u(s)+\partial_{3}G[u]_{\mathbb{R}}(s)u'(s)ds=0
\end{equation}
is equivalent to the second order differential equation
\begin{equation}
\label{eq 21}
\partial_{4}W\langle u\rangle_{\mathbb{R}}(t)u''(t)
+\partial_{3}W\langle u\rangle_{\mathbb{R}}(t) u'(t)
+\partial_{2}W\langle u\rangle_{\mathbb{R}}(t)u(t)=0.
\end{equation}
\end{proposition}

\begin{proof}
We show that coefficients of equations \eqref{eq 22} and \eqref{eq 21} are the same, respectively.
Let $\mathbb{T}=\mathbb{R}$. From equation \eqref{eq 23} and relation
$F\langle u\rangle_{\mathbb{R}} =\frac{d}{dt}H[u]_{\mathbb{R}}$ we have
$$
W(t,u(t),u'(t),u''(t))=\frac{d}{dt}H(t,u(t),u'(t))+G(t,u(t),u'(t)).
$$
Using notation \eqref{notation 1} and chain rule (that is valid for $\mathbb{T}=\mathbb{R}$ only)
we can calculate the partial derivatives:
\begin{itemize}
\item $\partial_{2}W\langle u\rangle_{\mathbb{R}}(t)
=\frac{d}{dt}\partial_{2}H[u]_{\mathbb{R}}(t)+\partial_{2}G[u]_{\mathbb{R}}(t)$,

\item $\partial_{3}W\langle u\rangle_{\mathbb{R}}(t)=\partial_{2}H[u]_{\mathbb{R}}(t)
+\frac{d}{dt}\partial_{3}H[u]_{\mathbb{R}}(t)+\partial_{3}G[u]_{\mathbb{R}}(t)$,

\item $\partial_{4}W\langle u\rangle_{\mathbb{R}}(t)=\partial_{3}H[u]_{\mathbb{R}}(t)$.
\end{itemize}
After differentiation both sides of equation \eqref{eq 22} we obtain
\begin{multline*}
\partial_{3}H[u]_{\mathbb{R}}(t) u''(t)
+\left(\partial_{2}H[u]_{\mathbb{R}}(t)
+\frac{d}{dt}\partial_{3}H[u]_{\mathbb{R}}(t)+\partial_{3}G[u]_{\mathbb{R}}(t)\right) u'(t)\\
+\left(\frac{d}{dt}\partial_{2}H[u]_{\mathbb{R}}(t)+\partial_{2}G[u]_{\mathbb{R}}(t)\right) u(t)=0.
\end{multline*}
Hence, the intended equivalence is proved.
\end{proof}

Proposition~\ref{prop} allows us to obtain the classical result of \cite[Theorem II]{Davis}
as a corollary of our Theorem~\ref{th self adj E-L}. The absence of a chain rule on time scales
(even for $\mathbb{T}=\mathbb{Z}$) implies that the classical approach of \cite{Davis}
fails on time scales. This is the reason why here we introduced a completely different approach
to the subject based on the integro-differential form.
The case $\mathbb{T}=\mathbb{Z}$ was recently investigated in \cite{Helmholtz}.
However, similarly to \cite{Davis}, the approach of \cite{Helmholtz}
is based on the differential form and cannot be extended to general time scales.


\section*{Acknowledgments}

This work was partially supported by Portuguese funds through the
\emph{Center for Research and Development in Mathematics and Applications} (CIDMA),
and \emph{The Portuguese Foundation for Science and Technology} (FCT),
within project PEst-OE/MAT/UI4106/2014.
Dryl was also supported by FCT through the Ph.D. fellowship
SFRH/BD/51163/2010; Torres by FCT within project OCHERA,
PTDC/EEI-AUT/1450/2012, co-financed by FEDER under POFC-QREN
with COMPETE reference FCOMP-01-0124-FEDER-028894.

The authors are very grateful to two anonymous referees
for valuable remarks and comments, which
significantly contributed to the quality of the paper.




\begin{thebibliography}{xx}

\bibitem{MR1066641}
S. Hilger,
Analysis on measure chains---a unified approach to continuous and discrete calculus,
Results Math. {\bf 18} (1990), no.~1-2, 18--56.

\bibitem{BohnerDEOTS}
M. Bohner\ and\ A. Peterson,
{\it Dynamic equations on time scales},
Birkh\"auser Boston, Boston, MA, 2001.

\bibitem{MR1908825}
R. Agarwal, M. Bohner, D. O'Regan\ and\ A. Peterson,
Dynamic equations on time scales: a survey,
J. Comput. Appl. Math. {\bf 141} (2002), no.~1-2, 1--26.

\bibitem{MBbook2003}
M. Bohner\ and\ A. Peterson,
{\it Advances in dynamic equations on time scales},
Birkh\"auser Boston, Boston, MA, 2003.

\bibitem{MyID:267}
M. Dryl, A. B. Malinowska\ and\ D. F. M. Torres,
A time-scale variational approach to inflation, unemployment and social loss,
Control Cybernet. {\bf 42} (2013), no.~2, 399--418.
{\tt arXiv:1304.5269}

\bibitem{MR3115452}
L. Bourdin\ and\ E. Tr\'elat,
Pontryagin maximum principle for finite dimensional
nonlinear optimal control problems on time scales,
SIAM J. Control Optim. {\bf 51} (2013), no.~5, 3781--3813.
{\tt arXiv:1302.3513}

\bibitem{MR3031162}
M. Dryl\ and\ D. F. M. Torres,
Necessary optimality conditions for infinite horizon variational problems on time scales,
Numer. Algebra Control Optim. {\bf 3} (2013), no.~1, 145--160.
{\tt arXiv:1212.0988}

\bibitem{MR2528200}
R. Almeida\ and\ D. F. M. Torres,
Isoperimetric problems on time scales with nabla derivatives,
J. Vib. Control {\bf 15} (2009), no.~6, 951--958.
{\tt arXiv:0811.3650}

\bibitem{MR3040923}
M. Dryl\ and\ D. F. M. Torres,
The delta-nabla calculus of variations for composition functionals on time scales,
Int. J. Difference Equ. {\bf 8} (2013), no.~1, 27--47.
{\tt arXiv:1211.4368}

\bibitem{MR2671876}
N. Martins\ and\ D. F. M. Torres,
Calculus of variations on time scales with nabla derivatives,
Nonlinear Anal. {\bf 71} (2009), no.~12, e763--e773.
{\tt arXiv:0807.2596}

\bibitem{Helmholtz}
L. Bourdin\ and\ J. Cresson,
Helmholtz's inverse problem of the discrete calculus of variations,
J. Difference Equ. Appl. {\bf 19} (2013), no.~9, 1417--1436.
{\tt arXiv:1203.1209}

\bibitem{Davis}
D. R. Davis,
The inverse problem of the calculus of variations in higher space,
Trans. Amer. Math. Soc. {\bf 30} (1928), no.~4, 710--736.

\bibitem{Cresson}
J. Cresson, A. B. Malinowska\ and\ D. F. M. Torres,
Time scale differential, integral, and variational embeddings of Lagrangian systems,
Comput. Math. Appl. {\bf 64} (2012), no.~7, 2294--2301.
{\tt arXiv:1203.0264}

\bibitem{TorresDeltaNabla}
D. F. M. Torres,
The variational calculus on time scales,
Int. J. Simul. Multidisci. Des. Optim. {\bf 4} (2010), no.~1, 11--25.
{\tt arXiv:1106.3597}

\bibitem{FerreiraMalinowskaTorres}
R. A. C. Ferreira, A. B. Malinowska\ and\ D. F. M. Torres,
Optimality conditions for the calculus of variations
with higher-order delta derivatives,
Appl. Math. Lett. {\bf 24} (2011), no.~1, 87--92.
{\tt arXiv:1008.1504}

\bibitem{FerreiraTorres}
R. A. C. Ferreira\ and\ D. F. M. Torres,
Necessary optimality conditions for the calculus of variations on time scales,
{\tt arXiv:0704.0656} [math.OC] (2007), 20~pp.

\bibitem{MR2004181}
B. van Brunt,
{\it The calculus of variations},
Universitext, Springer, New York, 2004.

\bibitem{MyID:179}
N. R. O. Bastos, R. A. C. Ferreira\ and\ D. F. M. Torres,
Discrete-time fractional variational problems,
Signal Process. {\bf 91} (2011), no.~3, 513--524.
{\tt arXiv:1005.0252}

\bibitem{Rahmat}
M. R. Segi Rahmat,
On some $(q,h)$-analogues of integral inequalities on discrete time scales,
Comput. Math. Appl. {\bf 62} (2011), no.~4, 1790--1797.

\bibitem{MyID:266}
A. B. Malinowska\ and\ D. F. M. Torres,
{\it Quantum variational calculus},
SpringerBriefs in Electrical and Computer Engineering: Control, Automation and Robotics,
Springer, New York, 2014.

\bibitem{MR2966852}
N. Martins\ and\ D. F. M. Torres,
Higher-order infinite horizon variational problems in discrete quantum calculus,
Comput. Math. Appl. {\bf 64} (2012), no.~7, 2166--2175.
{\tt arXiv:1112.0787}

\end{thebibliography}
\end{document}